%% file: sdimZ_arXiv_8.29.14.tex
\documentclass[11pt, notitlepage]{article}   

\usepackage{amsmath,amsthm,amsfonts}   

\usepackage{amscd}
\usepackage{amsfonts}
\usepackage{amssymb}
\usepackage{color}      
\usepackage{epsfig}
\usepackage{graphicx}           

\usepackage{epsfig}
\include{psfig}

\theoremstyle{plain}
\newtheorem{Thm}{Theorem}[section]
\newtheorem{Lem}[Thm]{Lemma}
\newtheorem{Crlr}[Thm]{Corollary}
\newtheorem{Prop}[Thm]{Proposition}
\newtheorem{Obs}[Thm]{Observation}
\newtheorem{Rem}[Thm]{Remark}

\theoremstyle{definition}

\newtheorem{Eg}{Example}

\def\finf{\mathop{{\rm I}\kern -.27 em {\rm F}}\nolimits}

\begin{document}

\date{}

\title{A Comparison between the Zero Forcing Number\\ and the Strong Metric Dimension of Graphs}

\author{{\bf Cong X. Kang}$^1$ and {\bf Eunjeong Yi}$^2$\\
\small Texas A\&M University at Galveston, Galveston, TX 77553, USA\\
{\small\em kangc@tamug.edu}$^1$; {\small\em yie@tamug.edu}$^2$}

\maketitle

\begin{abstract}
The \emph{zero forcing number}, $Z(G)$, of a graph $G$ is the minimum cardinality of a set $S$ of black vertices (whereas vertices in $V(G)-S$ are colored white) such that $V(G)$ is turned black after finitely many applications of ``the color-change rule": a white vertex is converted black if it is the only white neighbor of a black vertex. The \emph{strong metric dimension}, $sdim(G)$, of a graph $G$ is the minimum among cardinalities of all strong resolving sets: $W \subseteq V(G)$ is a \emph{strong resolving set} of $G$ if for any $u, v \in V(G)$, there exists an $x \in W$ such that either $u$ lies on an $x-v$ geodesic or $v$ lies on an $x-u$ geodesic. In this paper, we prove that $Z(G) \le sdim(G)+3r(G)$ for a connected graph $G$, where $r(G)$ is the cycle rank of $G$. Further, we prove the sharp bound $Z(G) \leq sdim(G)$ when $G$ is a tree or a unicyclic graph, and we characterize trees $T$ attaining $Z(T)=sdim(T)$. It is easy to see that $sdim(T+e)-sdim(T)$ can be arbitrarily large for a tree $T$; we prove that $sdim(T+e) \ge sdim(T)-2$ and show that the bound is sharp. 
\end{abstract}



\section{Introduction}

Let $G = (V(G),E(G))$ be a finite, simple, undirected, and connected graph of order $|V(G)| \ge 2$. The \emph{path cover number}, $P(G)$, of $G$ is the minimum number of vertex disjoint paths, occurring as induced subgraphs of $G$, that cover all the vertices of $G$. The \emph{degree} $\deg_G(v)$ of a vertex $v \in V(G)$ is the number of edges incident to the vertex $v$ in $G$; a \emph{leaf} (or \emph{pendant}) is a vertex of degree one. We denote the number of leaves of $G$ by $\sigma(G)$. For $S \subseteq V(G)$, we denote by $\langle S \rangle$ the subgraph induced by $S$. The \emph{distance} between two vertices $u, v \in V(G)$, denoted by $d_G(u, v)$, is the length of a shortest path in $G$ between $u$ and $v$. We omit $G$ when ambiguity is not a concern.\\

The notion of a zero forcing set, as well  as the associated zero forcing number, of a simple graph was introduced by the aforementioned ``AIM group" in~\cite{AIM} to bound the minimum rank of graphs. Let each vertex of a graph $G$ be given one of two colors, dubbed ``black" and ``white" by convention. Let $S$ denote the (initial) set of black vertices of $G$. The \emph{color-change rule} converts the color of a vertex from white to black if the white vertex $u_2$ is the only white neighbor of a black vertex $u_1$; we say ``$u_1$ forces $u_2$" in this case. The set $S$ is said to be \emph{a zero forcing set} of $G$ if all vertices of $G$ will be turned black after finitely many applications of the color-change rule. The \emph{zero forcing number}, $Z(G)$, of $G$ is the minimum of $|S|$, as $S$ varies over all zero forcing sets of $G$.\\ 

Since its introduction by the ``AIM group", zero forcing number has become a graph parameter studied for its own sake, as an interesting invariant of a graph. For example, for discussions on the number of steps it takes for a zero forcing set to turn the entire graph black (the graph parameter has been named the \emph{iteration index} or the \emph{propagation time} of a graph), see \cite{iteration} and \cite {proptime}. In~\cite{pzf}, a probabilistic interpretation of zero forcing in graphs is introduced. It's also noteworthy that physicists have independently studied the zero forcing parameter, referring to it as the \emph{graph infection number}, in conjunction with the control of quantum systems (see \cite{p1}, \cite{p2}, and \cite{p3}).\\ 

A vertex $x \in V(G)$ \emph{resolves} a pair of vertices $u,v \in V(G)$ if $d(u,x) \neq d(v,x)$. A vertex $x \in V(G)$ \emph{strongly resolves} a pair of vertices $u,v \in V(G)$ if $u$ lies on an $x-v$ geodesic or $v$ lies on an $x-u$ geodesic. A set of vertices $W \subseteq V(G)$ \emph{(strongly) resolves} $G$ if every pair of distinct vertices of $G$ is (strongly) resolved by some vertex in $W$; then $W$ is called a \emph{(strong) resolving set} of $G$. For an ordered set $W=\{w_1, w_2, \ldots, w_k\} \subseteq V(G)$ of distinct vertices, the \emph{metric representation} of $v \in V(G)$ with respect to $W$ is the $k$-vector $D_G(v | W)=(d(v, w_1), d(v, w_2), \ldots, d(v, w_k))$. The \emph{metric dimension} of $G$, denoted by $dim(G)$, is the minimum among cardinalities of all resolving sets of $G$. The \emph{strong metric dimension} of $G$, denoted by $sdim(G)$, is the minimum among cardinalities of all \emph{strong} resolving sets of $G$.\\

Metric dimension was introduced by Slater \cite{Slater} and, independently, by Harary and Melter \cite{HM}. Applications of metric dimension can be found in robot navigation \cite{landmarks}, sonar \cite{Slater}, combinatorial optimization \cite{MathZ}, and pharmaceutical chemistry \cite{CEJO}. Strong metric dimension was introduced by Seb\"{o} and Tannier \cite{MathZ}; they observed that if $W$ is a strong resolving set, then the vectors $\{D_G(v | W) \mid v \in V(G)\}$ uniquely determine the graph $G$ (also see \cite{fracsdim} for more detail); whereas for a resolving set $U$ of $G$, the vectors $\{D_G(v|U) \mid v \in V(G)\}$ may not uniquely determine $G$. It is noted that determining the (strong) metric dimension of a graph is an NP-hard problem (see \cite{NPcompleteness} and \cite{sdim}).\\

In this paper, we initiate a comparative study between the zero forcing number and the strong metric dimension of graphs. The zero forcing number and the strong metric dimension coincide for paths $P_n$, complete graphs $K_n$, complete bi-partite graphs $K_{s,t}$ ($s+t \ge 3$), for examples; they are $1$, $n-1$, and $s+t-2$, respectively. The Cartesian product of two paths shows that zero forcing number can be arbitrarily larger than strong metric dimension; cycles $C_n$ show that strong metric dimension can be arbitrarily larger than zero forcing number. We prove the sharp bound that $Z(G) \leq sdim(G)$ when $G$ is a tree or a unicyclic graph, and we characterize trees $T$ attaining $Z(T)=sdim(T)$. It is easy to see that $sdim(T+e)-sdim(T)$ can be arbitrarily large for a tree $T$; we prove that $sdim(T+e) \ge sdim(T)-2$ and show that the bound is sharp. In the final section, we show, for any graph $G$ with cycle rank $r(G)$, that $Z(G) \le sdim(G)+3r(G)$ and pose an open problem pertaining to its refinement.


\section{The zero forcing number and the strong metric dimension of trees}

In this section, we show that $Z(T) \le sdim(T)$ for a tree $T$, and we characterize trees $T$ satisfying $Z(T)=sdim(T)$. We first recall some results that will be used here.

\begin{Thm} \label{pathcover}
Let $T$ be a tree. Then
\begin{itemize}
\item[(a)] \cite{AIM} $Z(T) = P(T)$,
\item[(b)] \cite{MathZ} $sdim(T)=\sigma(T)-1$.
\end{itemize}
\end{Thm}

\begin{Thm}\cite{cutvertex} \label{cutV}
Let $G$ be a graph with cut-vertex $v \in V(G)$. Let $V_1, V_2, \ldots, V_k$ be the vertex sets for the connected components of $\langle V(G)- \{v\}\rangle$, and for $1\le i \le k$, let $G_i$ = $\langle V_i \cup \{v\}\rangle$. Then $Z(G) \ge  [\sum_{i=1}^{k} Z(G_i)]-k+1$.
\end{Thm}

The following terminology are defined for a graph $G$. A vertex of degree at least three is called a \emph{major vertex}. A leaf $u$ is called \emph{a terminal vertex of a major vertex} $v$ if $d(u, v)<d(u, w)$ for every other major vertex $w$. The \emph{terminal degree}, $ter(v)$, of a major vertex $v$ is the number of terminal vertices of $v$. A major vertex $v$ is an \emph{exterior major vertex} if it has positive terminal degree. An \emph{exterior degree two vertex} is a vertex of degree 2 that lies on a shortest path from a terminal vertex to its major vertex, and an \emph{interior degree two vertex} $z$ is a vertex of degree 2 such that a shortest path from $z$ to any terminal vertex includes a major vertex.

\begin{Thm} \cite{dimZ} \label{DZtree}
Let $T$ be a tree. Then
\begin{itemize}
\item[(a)] $dim(T) \leq Z(T)$,
\item[(b)] $dim(T)=Z(T)$ if and only if $T$ has no interior degree two vertex and each major vertex $v$ of $T$ satisfies $ter(v) \ge 2$.
\end{itemize}
\end{Thm}

It is shown in~\cite{dimZ2} that $P(T) \le \sigma(T)-1$; this and Theorem~\ref{pathcover} imply the following

\begin{Thm} \label{sdZtree}
For any tree $T$, $Z(T) \le sdim(T)$.
\end{Thm}

Next, we characterize trees $T$ satisfying $Z(T)=sdim(T)$.

\begin{Thm}\label{sdim=Z,T}
For any tree $T$, we have $Z(T)=sdim(T)$ if and only if $T$ has an interior degree two vertex on every $v_i-v_j$ path, where $v_i$ and $v_j$ are major vertices of $T$.
\end{Thm}

\begin{proof}
($\Longrightarrow$) Suppose that there exist a pair of major vertices, say $v_1$ and $v_2$, in $T$ such that no interior degree two vertex lies in the $v_1-v_2$ path. We may assume $v_1v_2 \in E(T)$. If not, replace $v_2$ with the vertex adjacent to $v_1$ on the $v_1-v_2$ path. We consider two disjoint subtrees $T_1,T_2 \subset T$ such that $v_1 \in V(T_1)$, $v_2 \in V(T_2)$, $V(T)=V(T_1) \cup V(T_2)$ and $E(T)=E(T_1) \cup E(T_2) \cup \{v_1v_2\}$. By Theorem \ref{sdZtree}, $P(T_1) \le \sigma(T_1)-1$ and $P(T_2) \le \sigma(T_2)-1$. So, $P(T) \le P(T_1)+P(T_2) \le \sigma(T_1)+\sigma(T_2)-2 =\sigma(T)-2$, i.e., $Z(T) \le sdim(T)-1$.\\

($\Longleftarrow$) We will induct on $m(T)$, the number of major vertices of the tree $T$. If $m(T)=0$, then $Z(T)=1=sdim(T)$; if $m(T)=1$, then $Z(T)=P(T)=\sigma(T)-1=sdim(T)$. Suppose the statement holds for all trees $T$ with $2\leq m(T)\leq k$. Let $x$ be a degree $2$ vertex lying between two major vertices $u$ and $v$ of a tree $T$ with $m(T)=k+1$. Let $\ell$ and $r$ be the two edges of $T$ incident with $x$, and denote by $T_\ell$ ($T_r$, resp.) the subtree of $T-r$ ($T-\ell$, resp.) containing $x$. Clearly, $T$ is the vertex sum of $T_\ell$ and $T_r$ at the vertices being labeled $x$. The induction hypothesis applies to $T_\ell$ and $T_r$, since each has at most $k$ major vertices; thus, $Z(T_\ell)=\sigma(T_\ell)-1$ and $Z(T_r)=\sigma(T_r)-1$. Now by Theorem \ref{cutV}, $Z(T)\geq (Z(T_\ell)+Z(T_r))-1=(\sigma(T_\ell)-1+\sigma(T_r)-1)-1=\sigma(T)-1= sdim(T)$; thus, by Theorem \ref{sdZtree}, $Z(T)=sdim(T)$.~\hfill
\end{proof}

\begin{Rem}
Notice $dim(T) \le Z(T) \le sdim(T)$ by Theorem~\ref{DZtree}(a) and Theorem~\ref{sdZtree}, where the equalities are characterized by Theorem~\ref{DZtree}(b) and Theorem~\ref{sdim=Z,T}.
\end{Rem}


\section{The zero forcing number and the strong metric dimension of unicyclic graphs}

A graph is \emph{unicyclic} if it contains exactly one cycle. Notice that a connected graph $G$ is unicyclic if and only if $|E(G)|=|V(G)|$. By $T+e$, we shall mean a unicyclic graph obtained from a tree $T$ by attaching the edge $e$ joining two non-adjacent vertices of $T$. In this section, we show that $Z(G) \le sdim(G)$ for a unicyclic graph $G$ and the bound is sharp. We first recall some results that will be used here.\\

We say that $x\in V(G)$ is \emph{maximally distant} from $y\in V(G)$ if $d_G(x,y)\geq d_G(z,y)$, for every $z\in N_G(x)=\{v\in V(G) \mid xv \in E(G)\}$. If $x$ is maximally distant from $y$ and $y$ is maximally distant from $x$, then we say that $x$ and $y$ are \emph{mutually maximally distant} and denote this by $x$ MMD $y$. It is pointed out in~\cite{sdim} that if $x$ MMD $y$ in $G$, then any strong resolving set of $G$ must contain either $x$ or $y$. Noting that any two distinct leaves of a graph $G$ are MMD, we have the following\\

\begin{Obs}\label{strongobservation2}
For any connected graph $G$, all but one of the $\sigma(G)$ leaves must belong to any strong resolving set of $G$.
\end{Obs}

\begin{Thm} \label{strongdimthm}
Let $G$ be a connected graph of order $n \ge 2$. Then
\begin{itemize}
\item[(a)] \cite{Z+e} $Z(G)-1 \le Z(G+e) \le Z(G)+1$ for $e \in E(\overline{G})$, where $\overline{G}$ denotes the complement of $G$,
\item[(b)] \cite{sdimGbar} $sdim(G)=1$ if and only if $G=P_n$.
\end{itemize}
\end{Thm}

\begin{Prop}\label{unicyclic}
Let $T$ be a tree of order at least three. Then $sdim(T+e) \ge sdim(T)-2$ for $e \in E(\overline{T})$, and the bound is sharp.
\end{Prop}

\begin{proof}
Since $\sigma(T)-2 \le \sigma(T+e) \le \sigma(T)$, the desired inequality follows from Theorem~\ref{pathcover}(b) and Observation~\ref{strongobservation2}. For the sharpness of the bound, let $T$ be the ``comb" with $k \ge 4$ exterior major vertices (see Figure~\ref{sduni}). Then $sdim(T)=\sigma(T)-1=k+1$. Since $\{\ell_i \mid 1 \le i \le k-1\}$ forms a strong resolving set for $T+e$, $sdim(T+e) \le k-1=sdim(T)-2$; thus $sdim(T+e)=sdim(T)-2$.~\hfill
\end{proof}

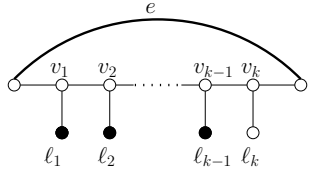
\begin{figure}[htpb]
\begin{center}
\scalebox{0.5}{\input{sduni.pstex_t}} \caption{Unicyclic graph $T+e$ satisfying $sdim(T+e)=sdim(T)-2$}\label{sduni}
\end{center}
\end{figure}

\begin{Rem}
We note that $sdim(T+e)-sdim(T)$ can be arbitrarily large. For example, suppose that $T=P_n$ and $T+e=C_n$; then $sdim(T)=1$ and, as noted in~\cite{sdim}, $sdim(C_n)=\lceil\frac{n}{2}\rceil$.
\end{Rem}

Theorem~\ref{sdZtree}, Theorem~\ref{strongdimthm}(a), and Proposition~\ref{unicyclic} imply that $Z(T+e) \le sdim(T+e)+3$. We will show that, in fact, $Z(T+e) \le sdim(T+e)$.\\

As defined in~\cite{tree-like}, a \emph{partial $n$-sun} is the graph $H_n$ obtained from $C_n$ by appending a leaf to each vertex in some $U \subseteq V(C_n)$, and a \emph{segment} of $H_n$ refers to any maximal subset of consecutive vertices in $U$. By a \emph{generalized partial $n$-sun}, we shall mean a graph obtained from $C_n$ by attaching a finite, and not necessarily equal, number of leaves to each vertex $v\in V(C_n)$. See Figure~\ref{6sun}.\\

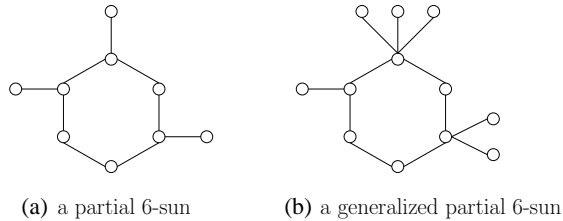
\begin{figure}[htpb]
\begin{center}
\scalebox{0.5}{\input{6sun.pstex_t}} \caption{A partial 6-sun and a generalized partial 6-sun}\label{6sun}
\end{center}
\end{figure}

\begin{Thm}~\cite{cutvertex}\label{segment}
Let $H_n$ be a partial $n$-sun with segments $U_1, U_2, \ldots, U_t$. Then
$$Z(H_n)=\max \left\{2, \sum_{i=1}^{t} \left\lceil \frac{|U_i|}{2} \right\rceil \right\}.$$
\end{Thm}

\begin{Crlr}\label{sdim_nsun}
Let $H_n$ be a partial $n$-sun. Then $Z(H_n) \le sdim(H_n)$.
\end{Crlr}

\begin{proof}
The formula in Theorem~\ref{segment} implies that $Z(H_n)\leq \left\lceil\frac{n}{2}\right\rceil$. Considering MMD vertices, it's clear that $sdim(H_n)\geq sdim(C_n)$ and, as noted in~\cite{sdim}, $sdim(C_n)=\lceil\frac{n}{2}\rceil$.~\hfill
\end{proof}

Following \cite{tree-like}, for a given unicyclic graph $G$, a vertex $v \in V(G)$ is called an \emph{appropriate vertex} if at least two components of $G-v$ are paths; a vertex $\ell \in V(G)$ is called a \emph{peripheral leaf} if $\deg_G(\ell)=1$, $\ell u \in E(G)$, and $\deg_G(u) = 2$ (whereas $\deg_G(u) \le 2$ in \cite{tree-like}). The \emph{trimmed form} of $G$ is an induced subgraph obtained by a sequence of deletions of appropriate vertices, isolated paths, and peripheral leaves until no more such deletions are possible. Further, define $sdim(G)=sdim(G_1)+sdim(G_2)$ (\emph{additivity of $sdim$ over disjoint components}), when $G$ is the disjoint union of $G_1$ and $G_2$. This is a natural extension of the (original) definition of $sdim$ for a connected graph; it is needed for the inductive arguments to come.\\

\begin{Rem} \cite{cutvertex} \label{z_trim}
Let $G$ be a unicyclic graph. Then
\begin{itemize}
\item[(a)] for an appropriate vertex $v$ in $G$, $Z(G-v)-1=Z(G)$;
\item[(b)] for an isolated path $P$ in $G$, $Z(G-V(P))+1=Z(G)$;
\item[(c)] for a peripheral leaf $\ell$ in $G$, $Z(G-\ell)=Z(G)$.
\end{itemize}
\end{Rem}

\begin{Lem} \label{sdim_trim}
Let $G$ be a unicyclic graph, and let $\mathcal{C}$ be the unique cycle in $G$.
\begin{itemize}
\item[(a)] If $v$ is an appropriate vertex in $G$ such that $v \not\in V(\mathcal{C})$, then $sdim(G-v)-1\leq sdim(G)$.
\item[(b)] If $P$ is an isolated path in $G$, then $sdim(G-V(P))+1=sdim(G)$.
\item[(c)] If $\ell$ is a peripheral leaf in $G$, then $sdim(G-\ell)=sdim(G)$.
\end{itemize}
\end{Lem}

\begin{proof} Let $\mathcal{M}_H(x)=\{y\in V(H): y \mbox{ MMD }x\}$.\\

(a) Denote the connected components of $G-v$ by $G_1$ (with $\mathcal{C}\subseteq G_1$) and $T_1,\ldots T_k$ ($k\geq 2$), of which $T_1$ and $T_2$ (and possibly more trees) are isolated paths; let $u$ denote the sole neighbor of $v$ in $V(G_1)$. Let $S$ be a minimum strong resolving set of $G$. \underline{Let $L$ denote the set of leaves in $G-G_1$.} By Observation~\ref{strongobservation2}, $0 \le |L-S| \le 1$. If $|L-S|=0$, then $S \cup \{u\}$ forms a strong resolving set for $G-v$, since a geodesic between any $\ell\in L$ and any $x\in V(G_1)$ necessarily passes through $u$; thus we have $sdim(G-v)-1\leq sdim(G)$. So, suppose $|L-S|=1$. Since $L$ strongly resolves the complement of $G_1$ in $G-v$, it suffices to prove the following  \\

\textbf{Claim.} $S\cap V(G_1)$ strongly resolves $G_1$.\\

\textit{Proof of Claim.} Let $\ell_0\in L-S$. Let $x, y \in V(G_1)$ be strongly resolved by $\ell\in L\cap S$; we will show that $x$ and $y$ are strongly resolved by some $z\in S \cap V(G_1)$. If $x$ or $y$, say $x$, does not lie on $\mathcal{C}$, then there must exist a leaf $\ell' \in V(G_1)\cap S$ which strongly resolves $x$ and $y$, and we are done. So, suppose both $x$ and $y$ lie on $\mathcal{C}$. Let $u'$ denote the vertex on $\mathcal{C}$ which is closest to $u$. There must exist a $w \in V(G_1)$ satisfying $w$ MMD $\ell_0$ and such that $d(u',w')$ equals the diameter of $\mathcal{C}$; here $w'$ denotes the vertex  on $\mathcal{C}$ which is closest to $w$. This $w$ lies in $S$, since $\ell_0\notin S$. Notice that $x$ and $y$ together lie on the same one of the two semi-circles defined by $u'$ and $w'$; otherwise, $u'-x$ geodesic does not contain $y$ and $u'-y$ geodesic does not contain $x$; the relevance here being that a geodesic from $\ell\in L$ to either $x$ or $y$ must pass through $u'$. Thus, without loss of generality, we may assume a $u'-y$ geodesic contains $x$. Then, a $w'-x$ geodesic, hence also a $w-x$ geodesic, contains $y$. It follows that $w\in S\cap V(G_1)$ strongly resolves $x$ and $y$.~$\Box$\\

(b) This follows from the fact $sdim(P)=1$ and the additivity of $sdim$ over disjoint components.\\

(c)  Since $\ell$ is a peripheral leaf in $G$, there exists a vertex $u \in V(G)$ such that $\ell u \in E(G)$ with $\deg_G(u) = 2$. Let $G'=G-\ell$. Since $\mathcal{M}_{G}(u)=\emptyset$ and $\mathcal{M}_{G'}(u)=\mathcal{M}_{G}(\ell)$, $sdim(G-\ell)=sdim(G)$.~\hfill
\end{proof}

\begin{Rem}
Let $G$ be a unicyclic graph, and let $\mathcal{C}$ be the unique cycle of $G$.
\begin{itemize}
\item[(a)] For an appropriate vertex $v \in V(G)$, $sdim(G)-sdim(G-v)$ can be arbitrarily large. If $G$ is a unicyclic graph as in (a) of Figure \ref{c3_rem}, then $sdim(G)=\lceil\frac{n}{2}\rceil+k-1$ and $sdim(G-v)=k+1$.
\item[(b)] There exists $G$ such that, for an appropriate vertex $v \in V(\mathcal{C})$, $sdim(G-v)=sdim(G)+2$. If $G$ is a unicyclic graph as in (b) of Figure \ref{c3_rem}, then $sdim(G)=6$ (the solid vertices form a minimum strong resolving set of $G$) and $sdim(G-v)=8$.
\end{itemize}
\end{Rem}

\begin{figure}[htpb]
\begin{center}
\scalebox{0.5}{\input{c3_rem.pstex_t}} \caption{Unicyclic graph $G$ and an appropriate vertex $v \in V(G)$}\label{c3_rem}
\end{center}
\end{figure}
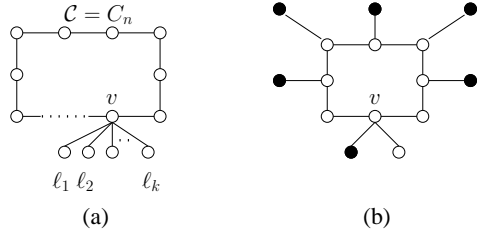

\begin{Lem}\label{gen-part-n-sun}
Let $H$ be a \textbf{generalized} partial $n$-sun. Then $Z(H) \le sdim(H)$.
\end{Lem}

\begin{proof}
It's clear that our claim holds for a $H$ which has only one major vertex. Thus, we may assume that $H$ contains at least two major vertices.
Let $H^0$ be a maximal partial $n$-sun contained in $H$; then $Z(H^0)\leq sdim(H^0)$ by Cororllary~\ref{sdim_nsun}. For $i\geq 0$, let $H^{i+1}$ denote the graph obtained as the vertex sum of a $P_2$ with $H^i$ at a major vertex of $H^i$, so that $H=H^k$ for some $k\geq 0$. By the choice of $H^0$, we have $sdim(H^{i+1})=sdim(H^i)+1\geq Z(H^{i})+1\geq Z(H^{i+1})$ for each $0\leq i\leq k-1$, where the left inequality is given by the induction hypothesis.~\hfill
\end{proof}

Now, we arrive at our main result.

\begin{Thm}
If $G$ is a unicyclic graph, then $Z(G) \le sdim(G)$.
\end{Thm}

\begin{proof}
Assume $Z(G)>sdim(G)$ for some unicyclic graph $G$.  By trimming as much as possible, but NOT trimming at any vertex lying on the unique cycle $\mathcal{C}$ of $G$, we arrive at a generalized partial $n$-sun $H\subseteq G$.
We descend from the given $G$ to $H$ by, for each trim at an allowed vertex $x$ of $G'$, discarding all components of $G'-x$ except the connected component $G''$ containing $\mathcal{C}$. Let $G'-x=G''+T_1+\ldots +T_m$, where $+$ denotes disjoint union. Remark~\ref{z_trim} and Lemma~\ref{sdim_trim} imply $Z(G''+T_1+\ldots +T_m)>sdim(G''+T_1+\ldots +T_m)$ which, by the additivity of both $Z$ and $sdim$, is equivalent to
\begin{equation}\label{descent}
Z(G'')+\sum_{i=1}^{m}Z(T_i)>sdim(G'')+\sum_{i=1}^{m}sdim(T_i).
\end{equation}
Since $Z(T_i)\leq sdim(T_i)$ for each tree $T_i$ by Theorem~\ref{sdZtree}, inequality (\ref{descent}) implies $Z(G'')>sdim(G'')$. Through this process of ``descent",  we eventually reach $Z(H)>sdim(H)$, which is the desired contradiction to Lemma~\ref{gen-part-n-sun}.~\hfill
\end{proof}

\begin{Rem}
There exists a unicyclic graph $G$ satisfying $Z(G)=sdim(G)$. For an odd integer $k \ge 3$, let $G$ be a partial $2k$-sun with the unique cycle $\mathcal{C}$ given by $u'_1u_2u'_3u_4 \ldots u'_{2k-1}u_{2k}$ such that $ter(u_{2j})=0$ and $ter(u'_{2j-1})=1$, where $1 \le j \le k$ (see Figure \ref{z=sd}). Then $Z(G)=k$ by Theorem \ref{segment}, and $sdim(G)=k$: (i) $sdim(G) \ge k$ since $u_j$ MMD $u_{j+k}$ for each $j \in \{1, 2, \ldots k\}$; (ii) $sdim(G) \le k$ since  $\{u_{2j-1} \mid 1 \le j \le k\}$ forms a strong resolving set for $G$.
\end{Rem}

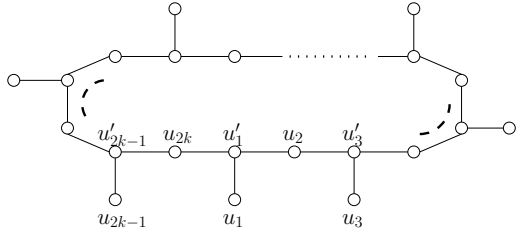
\begin{figure}[htpb]
\begin{center}
\scalebox{0.5}{\input{s_zd.pstex_t}} \caption{Unicyclic graphs $G$ with $Z(G)=sdim(G)$}\label{z=sd}
\end{center}
\end{figure}


\section{A concluding thought}

The cycle rank $r(G)$ of a connected graph $G$ is defined as $|E(G)|-|V(G)|+1$. In the preceding sections, we have provided \emph{sharp} bounds (relating $Z(G)$ and $sdim(G)$) when $r(G)$ equals $0$ or $1$; now, we offer a rough bound which, notably, places no restriction on $r(G)$.\\

\begin{Prop}
Let $G$ be a connected graph with cycle rank $r(G)$. Then $Z(G)\leq sdim(G)+3\cdot r(G)$.
\end{Prop}

\begin{proof}
Let $T$ be a spanning tree of $G$ obtained through the deletion of $r=r(G)$ edges of $G$. We have $Z(G)\leq Z(T)+r\leq sdim(T)+r$, where the left and right inequalities are respectively given by Theorem~\ref{strongdimthm}(a) and Theorem~\ref{sdZtree}. Since the removal of an edge $e$ from $G$ results in at most two more leaves in $G-e$, we have $\sigma(T)\leq 2r+\sigma(G)$. Since $sdim(T)=\sigma(T)-1$ by Theorem~\ref{pathcover}(b), we have $Z(G)\leq 2r+\sigma(G)-1+r$. Since $\sigma(G)-1\leq sdim(G)$ by Observation~\ref{strongobservation2}, we obtain $Z(G)\leq sdim(G)+3r$.~\hfill
\end{proof}

\noindent \textbf{Question.} What is the best $k$  such that $Z(G) \le sdim(G)+k\cdot r(G)$ for any connected graph $G$?\\

We conjecture $0<k<1$, as suggested by the following example.\\

\begin{Eg}
Let $G=P_s \square P_s$ be the Cartesian product of $P_s$ with itself, where $s \ge 2$. Then $Z(G)=s$ (see~\cite{AIM}) and $sdim(G)=2$. Notice that $r(G)=(s-1)^2$. So, $Z(G)=sdim(G)+\frac{s-2}{(s-1)^2} r(G)$. See Figure~\ref{33} when $s=3$, where the solid vertices in Figure~\ref{33}(a) form a minimum zero forcing set for $G$ and the solid vertices in Figure~\ref{33}(b) form a minimum strong resolving set for $G$.
\end{Eg}

\begin{figure}[htpb]
\begin{center}
\scalebox{0.5}{\input{sd33.pstex_t}} \caption{$Z(P_3 \square P_3)=3$ and $sdim(P_3 \square P_3)=2$}\label{33}
\end{center}
\end{figure}
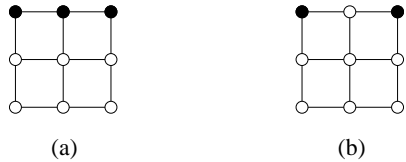


\end{document}

%% file: sduni.pstex_t
\begin{picture}(0,0)%
\includegraphics{sduni.pstex}%
\end{picture}%
\setlength{\unitlength}{3947sp}%
\begingroup\makeatletter\ifx\SetFigFont\undefined%
\gdef\SetFigFont#1#2#3#4#5{%
  \reset@font\fontsize{#1}{#2pt}%
  \fontfamily{#3}\fontseries{#4}\fontshape{#5}%
  \selectfont}%
\fi\endgroup%
\begin{picture}(3766,2194)(3518,-1436)
\put(3976,-1336){\makebox(0,0)[lb]{\smash{{\SetFigFont{17}{20.4}{\rmdefault}{\mddefault}{\updefault}{\color[rgb]{0,0,0}$\ell_1$}%
}}}}
\put(5251,539){\makebox(0,0)[lb]{\smash{{\SetFigFont{17}{20.4}{\rmdefault}{\mddefault}{\updefault}{\color[rgb]{0,0,0}$e$}%
}}}}
\put(6451,-211){\makebox(0,0)[lb]{\smash{{\SetFigFont{17}{20.4}{\rmdefault}{\mddefault}{\updefault}{\color[rgb]{0,0,0}$v_k$}%
}}}}
\put(5851,-211){\makebox(0,0)[lb]{\smash{{\SetFigFont{17}{20.4}{\rmdefault}{\mddefault}{\updefault}{\color[rgb]{0,0,0}$v_{k-1}$}%
}}}}
\put(4651,-211){\makebox(0,0)[lb]{\smash{{\SetFigFont{17}{20.4}{\rmdefault}{\mddefault}{\updefault}{\color[rgb]{0,0,0}$v_2$}%
}}}}
\put(4051,-211){\makebox(0,0)[lb]{\smash{{\SetFigFont{17}{20.4}{\rmdefault}{\mddefault}{\updefault}{\color[rgb]{0,0,0}$v_1$}%
}}}}
\put(6451,-1336){\makebox(0,0)[lb]{\smash{{\SetFigFont{17}{20.4}{\rmdefault}{\mddefault}{\updefault}{\color[rgb]{0,0,0}$\ell_k$}%
}}}}
\put(5851,-1336){\makebox(0,0)[lb]{\smash{{\SetFigFont{17}{20.4}{\rmdefault}{\mddefault}{\updefault}{\color[rgb]{0,0,0}$\ell_{k-1}$}%
}}}}
\put(4651,-1336){\makebox(0,0)[lb]{\smash{{\SetFigFont{17}{20.4}{\rmdefault}{\mddefault}{\updefault}{\color[rgb]{0,0,0}$\ell_2$}%
}}}}
\end{picture}%

%% file: 6sun.pstex_t
\begin{picture}(0,0)%
\includegraphics{6sun.pstex}%
\end{picture}%
\setlength{\unitlength}{3947sp}%
\begingroup\makeatletter\ifx\SetFigFont\undefined%
\gdef\SetFigFont#1#2#3#4#5{%
  \reset@font\fontsize{#1}{#2pt}%
  \fontfamily{#3}\fontseries{#4}\fontshape{#5}%
  \selectfont}%
\fi\endgroup%
\begin{picture}(7245,2729)(3518,-2407)
\put(4126,-2311){\makebox(0,0)[lb]{\smash{{\SetFigFont{17}{20.4}{\rmdefault}{\mddefault}{\updefault}{\color[rgb]{0,0,0}a partial 6-sun}%
}}}}
\put(7426,-2311){\makebox(0,0)[lb]{\smash{{\SetFigFont{17}{20.4}{\rmdefault}{\mddefault}{\updefault}{\color[rgb]{0,0,0}a generalized partial 6-sun}%
}}}}
\end{picture}%

%% file: c3_rem.pstex_t
\begin{picture}(0,0)%
\includegraphics{c3_rem.pstex}%
\end{picture}%
\setlength{\unitlength}{3947sp}%
\begingroup\makeatletter\ifx\SetFigFont\undefined%
\gdef\SetFigFont#1#2#3#4#5{%
  \reset@font\fontsize{#1}{#2pt}%
  \fontfamily{#3}\fontseries{#4}\fontshape{#5}%
  \selectfont}%
\fi\endgroup%
\begin{picture}(5866,2889)(3818,-3157)
\put(4501,-511){\makebox(0,0)[lb]{\smash{{\SetFigFont{17}{20.4}{\rmdefault}{\mddefault}{\updefault}{\color[rgb]{0,0,0}$\mathcal{C}=C_n$}%
}}}}
\put(5026,-1561){\makebox(0,0)[lb]{\smash{{\SetFigFont{17}{20.4}{\rmdefault}{\mddefault}{\updefault}{\color[rgb]{0,0,0}$v$}%
}}}}
\put(8326,-1561){\makebox(0,0)[lb]{\smash{{\SetFigFont{17}{20.4}{\rmdefault}{\mddefault}{\updefault}{\color[rgb]{0,0,0}$v$}%
}}}}
\put(5476,-2611){\makebox(0,0)[lb]{\smash{{\SetFigFont{17}{20.4}{\rmdefault}{\mddefault}{\updefault}{\color[rgb]{0,0,0}$\ell_k$}%
}}}}
\put(4651,-2611){\makebox(0,0)[lb]{\smash{{\SetFigFont{17}{20.4}{\rmdefault}{\mddefault}{\updefault}{\color[rgb]{0,0,0}$\ell_2$}%
}}}}
\put(4351,-2611){\makebox(0,0)[lb]{\smash{{\SetFigFont{17}{20.4}{\rmdefault}{\mddefault}{\updefault}{\color[rgb]{0,0,0}$\ell_1$}%
}}}}
\end{picture}%

%% file: s_zd.pstex_t
\begin{picture}(0,0)%
\includegraphics{s_zd.pstex}%
\end{picture}%
\setlength{\unitlength}{3947sp}%
\begingroup\makeatletter\ifx\SetFigFont\undefined%
\gdef\SetFigFont#1#2#3#4#5{%
  \reset@font\fontsize{#1}{#2pt}%
  \fontfamily{#3}\fontseries{#4}\fontshape{#5}%
  \selectfont}%
\fi\endgroup%
\begin{picture}(6391,2879)(2843,-3157)
\put(3976,-3061){\makebox(0,0)[lb]{\smash{{\SetFigFont{17}{20.4}{\rmdefault}{\mddefault}{\updefault}{\color[rgb]{0,0,0}$u_{2k-1}$}%
}}}}
\put(7051,-2011){\makebox(0,0)[lb]{\smash{{\SetFigFont{17}{20.4}{\rmdefault}{\mddefault}{\updefault}{\color[rgb]{0,0,0}$u'_3$}%
}}}}
\put(6301,-2011){\makebox(0,0)[lb]{\smash{{\SetFigFont{17}{20.4}{\rmdefault}{\mddefault}{\updefault}{\color[rgb]{0,0,0}$u_2$}%
}}}}
\put(5551,-2011){\makebox(0,0)[lb]{\smash{{\SetFigFont{17}{20.4}{\rmdefault}{\mddefault}{\updefault}{\color[rgb]{0,0,0}$u'_1$}%
}}}}
\put(4801,-2011){\makebox(0,0)[lb]{\smash{{\SetFigFont{17}{20.4}{\rmdefault}{\mddefault}{\updefault}{\color[rgb]{0,0,0}$u_{2k}$}%
}}}}
\put(3976,-2011){\makebox(0,0)[lb]{\smash{{\SetFigFont{17}{20.4}{\rmdefault}{\mddefault}{\updefault}{\color[rgb]{0,0,0}$u'_{2k-1}$}%
}}}}
\put(7051,-3061){\makebox(0,0)[lb]{\smash{{\SetFigFont{17}{20.4}{\rmdefault}{\mddefault}{\updefault}{\color[rgb]{0,0,0}$u_3$}%
}}}}
\put(5551,-3061){\makebox(0,0)[lb]{\smash{{\SetFigFont{17}{20.4}{\rmdefault}{\mddefault}{\updefault}{\color[rgb]{0,0,0}$u_1$}%
}}}}
\end{picture}%

%% file: sd33.pstex_t
\begin{picture}(0,0)%
\includegraphics{sd33.pstex}%
\end{picture}%
\setlength{\unitlength}{3947sp}%
\begingroup\makeatletter\ifx\SetFigFont\undefined%
\gdef\SetFigFont#1#2#3#4#5{%
  \reset@font\fontsize{#1}{#2pt}%
  \fontfamily{#3}\fontseries{#4}\fontshape{#5}%
  \selectfont}%
\fi\endgroup%
\begin{picture}(4966,1979)(4118,-2257)
\end{picture}%

%% file: sdimZ_arXiv_8.29.14.bbl
\begin{thebibliography}{99}

\bibitem{AIM} F. Barioli, W. Barrett, S. Butler, S.M. Cioab\u{a}, D. Cvetkovi\'{c}, S.M. Fallat, C. Godsil, W. Haemers, L. Hogben, R. Mikkelson, S. Narayan, O. Pryporova, I. Sciriha, W. So, D. Stevanovi\'{c}, H. van der Holst, K. Vander Meulen and A.W. Wehe (AIM Minimum Rank-Special Graphs Work Group), Zero forcing sets and the minimum rank of graphs. {\it Linear Algebra Appl.} {\bf{428}} (2008) 1628-1648.

\bibitem{tree-like} F. Barioli, S. Fallat and L. Hogben, On the difference between the maximum multiplicity and path cover number for tree-like graphs. \emph{Linear Algebra Appl.} \textbf{409} (2005) 13-31.

\bibitem{p1} D. Burgarth and V. Giovannetti, Full control by locally induced relaxation. \emph{Phys. Rev. Lett.} \textbf{99} (2007) 100501.

\bibitem{p2} D. Burgarth and K. Maruyama, Indirect Hamiltonian identification through a small gateway. \emph{New J. Phys.} \textbf{11} (2009) 103019.

\bibitem{CEJO} G. Chartrand, L. Eroh, M.A. Johnson and O.R. Oellermann, Resolvability in graphs and the metric dimension of a graph. \textit{Discrete Appl. Math.} 
\textbf{105} (2000) 99-113.

\bibitem{iteration} K. Chilakamarri, N. Dean, C.X. Kang and E. Yi, Iteration index of a zero forcing set in a graph. \textit{Bull. Inst. Combin. Appl.} {\bf{64}} (2012) 57-72.

\bibitem{Z+e} C.J. Edholm, L. Hogben, M. Huynh, J. LaGrange and D.D. Row, Vertex and edge spread of zero forcing number, maximum nullity, and minimum rank of a graph. \textit{Linear Algebra Appl.} \textbf{436} (2012) 4352-4372.

\bibitem{dimZ} L. Eroh, C.X. Kang and E. Yi, A comparison between the metric dimension and zero forcing number of trees and unicyclic graphs. \textit{arXiv:1408.5943}.

\bibitem{dimZ2} L. Eroh, C.X. Kang and E. Yi, Metric dimension and zero forcing number of two families of line graphs. \textit{Math. Bohem.}, to appear.

\bibitem{NPcompleteness} M.R. Garey and D.S. Johnson, \emph{Computers and intractability: A guide to the theory of NP-completeness.} Freeman, New York (1979).

\bibitem{HM} F. Harary and R.A. Melter, On the metric dimension of a graph. \textit{Ars Combin.} {\bf{2}} (1976) 191-195.

\bibitem{proptime} L. Hogben, M. Huynh, N. Kingsley, S. Meyer, S. Walker and M. Young, Propagation time for zero forcing on a graph. \textit{Discrete Appl. Math.} \textbf{160} (2012) 1994-2005.

\bibitem{pzf} C.X. Kang and E. Yi, Probabilistic zero forcing in graphs. \textit{Bull. Inst. Combin. Appl.} \textbf{67} (2013) 9-16.

\bibitem{fracsdim} C.X. Kang and E. Yi, The fractional strong metric dimension of graphs. COCOA'13, \textit{Lecture Notes in Comput. Sci.} \textbf{8287} (2013) 84-95.

\bibitem{landmarks} S. Khuller, B. Raghavachari and A. Rosenfeld, Landmarks in graphs. \textit{Discrete Appl. Math.} \textbf{70} (1996) 217-229.

\bibitem{sdim} O.R. Oellermann and J. Peters-Fransen, The strong metric dimension of graphs and digraphs. \textit{Discrete Appl. Math.} \textbf{155} (2007) 356-364.

\bibitem{cutvertex} D.D. Row, A technique for computing the zero forcing number of a graph with a cut-vertex. \emph{Linear Algebra Appl.} \textbf{436} (2012) 4423-4432.

\bibitem{MathZ} A. Seb\"{o} and E. Tannier, On metric generators of graphs. \emph{Math. Oper. Res.} \textbf{29} (2004) 383-393.

\bibitem{p3} S. Severini, Nondiscriminatory propagation on trees. \emph{J. Phys. A: Math. Theor.} \textbf{41} (2008) 482002.

\bibitem{Slater} P.J. Slater, Leaves of trees. \textit{Congr. Numer.} {\bf{14}} (1975) 549-559.

\bibitem{sdimGbar} E. Yi, On strong metric dimension of graphs and their complements. \textit{Acta Math. Sin. (Engl. Ser.)} \textbf{29} (2013) 1479-1492.

\end{thebibliography}
